\documentclass[10pt,reqno]{amsart}
\usepackage{float}
\usepackage{graphicx}
\usepackage{graphics}
\usepackage{subfigure}
\usepackage{mathrsfs}
\usepackage[colorlinks,linkcolor=black,anchorcolor=black,citecolor=black,hyperindex=black,CJKbookmarks=black]{hyperref}
\usepackage{mathabx}
\newtheorem{theorem}{Theorem}[section]
\newtheorem{lemma}[theorem]{Lemma}

\theoremstyle{definition}

\numberwithin{equation}{section}

\begin{document}

\title[]{On upper and lower fast Knintchine spectra of continued fractions}%

\author {Lulu Fang}
\address{School of Science, Nanjing University of Science and Technology, Nanjing 210094, China}
\email{fanglulu@njust.edu.cn}

\author{Lei Shang}
\address{School of Mathematics, South China University of Technology, Guangzhou 510640, China}
\email{auleishang@gmail.com}

\author {Min Wu}
\address{School of Mathematics, South China University of Technology, Guangzhou 510640, China}
\email{wumin@scut.edu.cn}


\begin{abstract}
Let $\psi:\mathbb{N}\to \mathbb{R}^+$ be a function satisfying $\phi(n)/n\to \infty$ as $n \to \infty$. We investigate from a multifractal analysis point of view the growth rate of the sums $\sum^n_{k=1}\log a_k(x)$ relative to $\psi(n)$, where $[a_1(x),a_2(x), a_3(x)\cdots]$ denotes the continued fraction expansion of $x\in (0,1)$. The upper (resp. lower) fast Khintchine spectrum is defined by the Hausdorff dimension of the set of all points $x$ for which the upper (resp. lower) limit of $\frac{1}{\psi(n)}\sum^n_{k=1}\log a_k(x)$ is $1$.
The precise formulas of these two spectra are completely determined, which strengthens a result of Liao and Rams (2016).

\end{abstract}

\maketitle

\section{Introduction}

Let $G:[0,1) \to [0,1)$ be the \emph{Gauss map}, defined as $G(0):=0$ and
\[
G(x):= 1/x - \lfloor1/x\rfloor,\ \ \ \forall x \in (0,1),
\]
where $\lfloor\cdot\rfloor$ stands for the integer part of a number. For $x\in (0,1)$, put $a_1(x):=\lfloor1/x\rfloor$ and
$a_n(x):=a_1(G^{n-1}(x))$ for $n \geq 2$, where $G^k$ denotes the $k$th iteration of $G$. Then $x$ admits a unique \emph{continued fraction expansion} of the form
\begin{equation}\label{CF}
x = \dfrac{1}{a_1(x) +\dfrac{1}{a_2(x)  +\dfrac{1}{a_3(x)+\ddots}}}
\end{equation}
where $a_1(x),a_2(x),a_3(x),\cdots$ are positive integers, and are called the \emph{partial quotients} of $x$.
See \cite{IK02, Khin64} for more information of continued fractions.

The \emph{Khintchine exponent} of $x$ is defined as the growth rate of the geometric average of partial quotients, namely,
\[
\mathrm{k}(x):=\lim_{n\to \infty} \frac{\log a_1(x)+\cdots+\log a_n(x)}{n}
\]
if the limit exists. Khintchine \cite{Khin} proved that for Lebesgue almost all $x\in (0,1)$, $\mathrm{k}(x) = \log c$, where $c=2.6854...$ is called the \emph{Khintchine constant}, see \cite{BBC} for details.
Furthermore, Fan et al. \cite{FLWW09} studied from a multifractal analysis point of view the set of points with a given Khintchine exponent that is different from $\log c$, and showed that the dimensional function (called the \emph{Khintchine spectrum})
\[
[0,\infty] \ni \alpha \mapsto K(\alpha):=\dim_{\rm H}\big\{x\in(0,1): \mathrm{k}(x) =\alpha\big\},
\]
is real analytic, strictly increasing in $[0, \log c)$ and strictly decreasing in $[\log c, \infty)$; while it is neither convex nor concave, see \cite{FJLR, IJ} for general results.
We point out that $K(\infty) =1/2$, see for example \cite[Theorem 7.1]{IJ}.
This means that there are uncountably many points with infinite Khintchine exponent and then leads to the question of detailed analyses of numbers with infinite Khintchine exponent.

For this purpose, let $\psi:\mathbb{N}\to \mathbb{R}^+$ be a function satisfying $\phi(n)/n\to \infty$ as $n \to \infty$ and write
\[
E(\psi):=\left\{x\in (0,1):\lim_{n\to \infty} \frac{\log a_1(x)+\cdots+\log a_n(x)}{\psi(n)}=1\right\}.
\]
The Hausdorff dimension of $E(\psi)$ is called the \emph{fast Khintchine spectrum}, and then completely determined by Fan et al. \cite{FLWW13}.

\begin{theorem}[{\cite[Theorem 1.1]{FLWW13}}]\label{lim}
Assume that $\psi:\mathbb{N}\to \mathbb{R}^+$ is non-decreasing. Then
\[
\dim_{\rm H}E(\psi) = \frac{1}{\beta+1}, \ \ \ \text{with}\ \ \beta:=\limsup_{n\to \infty} \frac{\psi(n+1)}{\psi(n)}.
\]
\end{theorem}

We remark that the assumption on the monotonicity of $\psi$ is necessary. In fact, $E(\psi)$ is non-empty if and only if $\psi$ has a monotonicity in a certain sense, see \cite[Lemma 3.1]{FLWW13} for precise statements.
Indeed, the limit in $E(\psi)$ means that $\psi(n)$ should keep pace with the sum $\log a_1(x)+\cdots+\log a_n(x)$ for sufficiently large $n$, but the sum is non-decreasing, so $\psi(n)$ must be ultimately ``increasing".

In the present paper, we are concerned with the following sets:
\[
\overline{E}(\psi):=\left\{x\in (0,1): \limsup_{n\to \infty} \frac{\log a_1(x)+\cdots+\log a_n(x)}{\psi(n)}=1\right\}
\]
and
\[
\underline{E}(\psi):=\left\{x\in (0,1): \liminf_{n\to \infty} \frac{\log a_1(x)+\cdots+\log a_n(x)}{\psi(n)}=1\right\},
\]
where $\psi$ is as defined above. Then $E(\psi)=\overline{E}(\psi) \cap\underline{E}(\psi)$. The Hausdorff dimensions of $\overline{E}(\psi)$ and $\underline{E}(\psi)$ are called \emph{upper} and \emph{lower fast Khintchine spectra} respectively. Unlike the set $E(\psi)$, we will show below that $\overline{E}(\psi)$ and $\underline{E}(\psi)$ are always non-empty, see Lemmas \ref{overiff} and \ref{underiff}.
These two spectra have been studied by Liao and Rams \cite[Theorem 1.2]{LR} under some restrictions on the growth rate of $\psi$ (i.e., $b,B>1$, see below for their definitions).
Roughly speaking, their results can be applied to functions with fast growth speed (e.g., exponential functions), but can not be used to polynomial functions.
However, for $\psi(n)=n^p$ $(p>1)$, it follows from Theorem \ref{lim} that the upper and lower fast Khintchine spectra are not less than $1/2$; applying standard covering arguments, we see that these two spectra not greater than $1/2$, see for example \cite{FMS, JM, WW}. Hence the Hausdorff dimensions of $\overline{E}(\psi)$ and $\underline{E}(\psi)$ are $1/2$ for polynomial functions $\psi$, which indicates that the results of Liao and Rams may be true for the cases $b=1$ and $B=1$. This is precisely what we are covering in the present paper by giving the upper and lower fast Khintchine spectra without any extra assumptions on $\psi$.

\begin{theorem}\label{OU}
Let $\psi:\mathbb{N}\to \mathbb{R}^+$ be defined above. Then
\[
\dim_{\rm H}\overline{E}(\psi) = \frac{1}{b+1}\ \ \text{and}\ \ \dim_{\rm H}\underline{E}(\psi) = \frac{1}{B+1},
\]
where $b, B \in [1,\infty]$ are given by
\[
\log b:=\liminf_{n \to \infty}\frac{\log \psi(n)}{n} \ \ \text{and}\ \ \log B:=\limsup_{n \to \infty}\frac{\log \psi(n)}{n}.
\]

\end{theorem}


As mentioned above, Liao and Rams \cite{LR} dealt with the cases $b, B >1$.
However, our approach works well for all cases of $b$ and $B$ instead of the remaining special cases $b=1$ and $B=1$.
Comparing with the methods of Liao and Rams, we no longer treat $\log a_1(x)+\cdots+\log a_n(x)$ as a sum, but as a logarithm of the product of partial quotients, i.e., $\log (a_1(x)\cdots a_n(x))$. This is the difference between us and Liao and Rams, and is also the main advantage of our method.
The little change helps to establish powerful results (see Lemmas \ref{F} and \ref {Fwidehat}) for calculating the upper bounds of $\dim_{\rm H}\overline{E}(\psi)$ and $\dim_{\rm H}\underline{E}(\psi)$,
which is slightly better than that of Liao and Rams since their proofs of this part heavily rely on $b, B >1$, see \cite[p.\,71\,\&\,75]{LR}.
For the lower bounds of $\dim_{\rm H}\overline{E}(\psi)$ and $\dim_{\rm H}\underline{E}(\psi)$, we also consider the product of partial quotients as a whole, and then use Lemma \ref{2S} to obtain them by constructing suitable sequences.

Liao and Rams \cite{LR} remarked that $b\leq B\leq \beta$, and one can construct some $\psi$ such that the values of $b,B$ and $\beta$ are all different. Here we give a concrete example to show this.
For $k \geq 1$, let $n_k:=1!+2!+\cdots+k!$ and
\[
\psi(n):=
\left\{
  \begin{array}{ll}
    5^{k-1}4^{n-(1!+3!+\cdots+(2k-1)!)}3^{1!+3!+\cdots+(2k-1)!}, & \hbox{$n_{2k-1}<n\leq n_{2k}$;} \\
    5^{k}4^{2!+4!+\cdots+(2k)!}3^{n- (2!+4!+\cdots+(2k)!)},& \hbox{$n_{2k}<n \leq n_{2k+1}$.}
  \end{array}
\right.
\]
Then $b=3, B=4$ and $\beta=15$ since
\[
\lim_{k \to \infty} \frac{\sum^k_{j=1}(2j)!}{\sum^{2k}_{j=1}j!} = \lim_{k \to \infty} \frac{\sum^k_{j=1}(2j+1)!}{\sum^{2k+1}_{j=1}j!} =1 \ \ \ \text{and}\ \ \ \lim_{k \to \infty} \frac{\psi(n_{2k}+1)}{\psi(n_{2k})}=15.
\]
In other words, these three fast Khintchine spectra can be all different, which is a new phenomenon in continued fractions and in symbolic systems with countably many symbols, since the multifractal spectra in symbolic systems with finitely many symbols are always the same when limit is replaced by limsup or liminf, see \cite{BRS, Bes, Egg, FSS} more information.

\section{Proof of Theorem \ref{OU}}
Before proving Theorem \ref{OU}, we first give several useful lemmas. The first result is due to {\L}uczak \cite{Luc97}, see also \cite{JM, WW} for general results.

\begin{lemma}[\cite{Luc97}]\label{Luc}
Let $a, c \in (1,\infty)$. Then
\begin{align*}
\dim_{\rm H}&\left\{x\in (0,1): a_n(x) \geq a^{c^n},\forall n\geq 1\right\} \\
=& \dim_{\rm H}\left\{x\in (0,1): a_n(x) \geq a^{c^n},\text{i.m.}\ n \in \mathbb{N}\right\} = \frac{1}{c+1},
\end{align*}
where ``\text{i.m.}" denotes ``infinitely many".
\end{lemma}

Note that the sum $\log a_1(x)+\cdots+\log a_n(x)$ is considered as $\log (a_1(x)\cdots a_n(x))$, so we need to study the product of partial quotients.
Write $\Pi_n(x):=a_1(x)\cdots a_n(x)$. The following result of $\Pi_n(x)$ is analogous to Lemma \ref{Luc}.

\begin{lemma}\label{Pac}
Let $a, c \in (1,\infty)$. Then
\begin{align*}
\dim_{\rm H}&\left\{x\in (0,1): \Pi_n(x) \geq a^{c^n},\forall n\gg 1\right\} \\
=& \dim_{\rm H}\left\{x\in (0,1): \Pi_n(x) \geq a^{c^n},\text{i.m.}\ n \in \mathbb{N}\right\} = \frac{1}{c+1},
\end{align*}
where ``$\forall n\gg 1$" denotes ``for sufficiently large $n$".
\end{lemma}

\begin{proof}
It is sufficient to estimate the lower bound of the Hausdorff dimension of the first set and the upper bound for the second set.
Since $\Pi_n(x) \geq a_n(x)$, by Lemma \ref{Luc}, we deduce that
\[
\dim_{\rm H} \left\{x\in (0,1): \Pi_n(x) \geq a^{c^n},\forall n\gg 1\right\} \geq \frac{1}{c+1}.
\]
Let $0<\varepsilon <c-1$ be fixed. We claim that
\begin{equation}\label{subset}
\left\{x\in (0,1):\Pi_n(x) \geq a^{c^n},\text{i.m.}\ n \in \mathbb{N}\right\} \subseteq \left\{x\in (0,1): a_n(x) \geq a^{(c-\varepsilon)^n},\text{i.m.}\ n \in \mathbb{N}\right\}.
\end{equation}
In fact, if there exists $N:=N(x,\varepsilon)\in \mathbb{N}$ such that $a_n(x) < a^{(c-\varepsilon)^n}$ for all $n > N$, then
\[
\Pi_n(x) <\Pi_N(x)a^{(c-\varepsilon)^{N+1}+\cdots+(c-\varepsilon)^n}<\Pi_N(x)a^{\frac{(c-\varepsilon)^{n+1}}{c-\varepsilon-1}},
\]
and hence $\Pi_n(x) <a^{c^n}$ for sufficiently large $n$.
It follows from \eqref{subset} and Lemma \ref{Luc} that
\[
\dim_{\rm H}\left\{x\in (0,1): \Pi_n(x) \geq a^{c^n},\text{i.m.}\ n \in \mathbb{N}\right\} \leq \frac{1}{c-\varepsilon+1}.
\]
Letting $\varepsilon \to 0^+$, we obtain the desired upper bound.
\end{proof}

We point out that $K(\infty)=1/2$, see \cite[Theorem 1.2]{FLWW09} and \cite[Theorem 7.1]{IJ}. However,
the following lemma shows that the set of points for which the limsup of their geometric averages is infinity is also of Hausdorff dimension $1/2$.

\begin{lemma}\label{P}
Write
\[
\Pi_\infty:=\left\{x\in (0,1): \limsup_{n\to\infty}\frac{\log \Pi_n(x)}{n}=\infty\right\}.
\]
Then
\[
\dim_{\rm H}\Pi_\infty =\frac{1}{2}.
\]
\end{lemma}

\begin{proof}
For the lower bound, since
\[
\Pi_\infty  \supseteq \left\{x\in (0,1): \Pi_n(x) \geq a^{c^n},\forall n\gg 1\right\},
\]
by Lemma \ref{Pac}, we see that $\dim_{\rm H}\Pi_\infty  \geq 1/(c+1)$, and so $\dim_{\rm H}\Pi_\infty  \geq 1/2$.

For the upper bound, let $0<\varepsilon <1$ and $s:= 1/2 +\varepsilon$. Choose a sufficiently large number $K>1$ such that
\[
K^\varepsilon > 2 M_\varepsilon
\]
where $M_\varepsilon$ is defined as $M_\varepsilon:= \sum_{j\geq 1}j^{-(1+\varepsilon)}$. For $(\sigma_1,\cdots,\sigma_n) \in \mathbb{N}^n$, the set
\[
I_n(\sigma_1,\cdots,\sigma_n):=\big\{x\in (0,1): a_k(x)=\sigma_k\ \text{for all}\ 1\leq k \leq n\big\}
\]
is called a \emph{cylinder} of order $n$ associated to $(\sigma_1,\cdots,\sigma_n)$. It was shown in \cite[p.\,18]{IK02} that $I_n(\sigma_1,\cdots,\sigma_n)$ is an interval and
its length satisfies
\begin{equation}\label{length}
|I_n(\sigma_1,\cdots,\sigma_n)| \leq \frac{1}{(\sigma_1 \cdots \sigma_n)^2}.
\end{equation}
Denote by $\mathcal{H}^s(\cdot)$ the $s$-dimensional Hausdorff measure. Remark that $\Pi_\infty$ is covered by
\begin{align*}
\bigcap^\infty_{N=1} \bigcup^\infty_{n=N} \bigcup_{(\sigma_1,\cdots,\sigma_n)\in \mathcal{C}_n(K)} I_n(\sigma_1,\cdots,\sigma_n),
\end{align*}
where $\mathcal{C}_n(K)$ is given by $\mathcal{C}_n(K):= \{(\sigma_1,\cdots,\sigma_n) \in \mathbb{N}^n: \sigma_1 \cdots\sigma_n \geq K^n\}$.
By \eqref{length}, we conclude that
\begin{align*}
\mathcal{H}^s(\Pi_\infty) &\leq \liminf_{N\to \infty} \sum^\infty_{n =N} \sum_{(\sigma_1,\cdots,\sigma_n)\in \mathcal{C}_n(K)} |I_n(\sigma_1,\cdots,\sigma_n)|^s\\
& \leq \liminf_{N\to \infty} \sum^\infty_{n =N} \sum_{(\sigma_1,\cdots,\sigma_n)\in \mathcal{C}_n(K)} \frac{1}{(\sigma_1 \cdots \sigma_n)^{1+2\varepsilon}}\\
&\leq \liminf_{N\to \infty} \sum^\infty_{n =N} \frac{1}{K^{\varepsilon n}}\sum_{(\sigma_1,\cdots,\sigma_n)\in \mathbb{N}^n} \frac{1}{(\sigma_1 \cdots \sigma_n)^{1+\varepsilon}}\\
& \leq \liminf_{N\to \infty} \sum^\infty_{n =N} \left(\frac{M_\varepsilon}{K^{\varepsilon}}\right)^n\\
& \leq \liminf_{N\to \infty} \sum^\infty_{n =N} \frac{1}{2^n} =0,
\end{align*}
which implies that $\dim_{\rm H}\Pi_\infty \leq s$. Since $\varepsilon$ is arbitrary, we obtain $\dim_{\rm H}\Pi_\infty \leq 1/2$. The proof is complete.
\end{proof}

The following result, given by Fan et al. \cite{FLWW09, FLWW13}, provides a powerful method for estimating the lower bound of the Hausdorff dimension of certain sets arising in continued fraction expansions. See Liao and Rams \cite{LR19} for a general result in the setting of infinite iterated function systems.

\begin{lemma}[\cite{FLWW09, FLWW13}]\label{2S}
Let $\{s_n\}$ be a sequence of real numbers with $s_n \geq 1$ for all $n \geq 1$ and
\begin{equation}\label{sinfty}
\lim_{n \to \infty}\frac{\sum^n_{k=1}\log s_k}{n}=\infty.
\end{equation}
Write
\[
\mathbb{E}(\{s_n\}):= \big\{x\in (0,1): s_n \leq a_n(x)\leq 2s_n, \forall n \geq 1\big\}.
\]
Then
\[
\dim_{\rm H}\mathbb{E}(\{s_n\}) = \left(2+\limsup_{n \to \infty}\frac{\log s_{n+1}}{\log s_1+\cdots+\log s_n}\right)^{-1}.
\]
\end{lemma}

In what follows, we will calculate the Hausdorff dimensions of $\overline{E}(\psi)$ and $\underline{E}(\psi)$ respectively.

\subsection{Hausdorff dimension of $\overline{E}(\psi)$}
Let $\varphi, \varphi^\prime: \mathbb{N} \to \mathbb{R}^+$ be functions. We say that $\varphi$ is \emph{limsup-equivalent} to $\varphi^\prime$ if
\[
\limsup_{n \to \infty} \frac{\varphi^\prime(n)}{\varphi(n)}=1.
\]
Similarly, $\varphi$ is said to be \emph{liminf-equivalent} to $\varphi^\prime$ if limsup is replaced by liminf. Note that $\varphi$ is liminf-equivalent to $\varphi^\prime$ if and only if $\varphi^\prime$ is limsup-equivalent to $\varphi$.

Recall that $\psi:\mathbb{N} \to \mathbb{R}^+$ is a function satisfying $\psi(n)/n \to \infty$ as $n \to \infty$ (without any monotonicity).
We will show that $\overline{E}(\psi)$ is always non-empty. To this end, we first give a necessary and sufficient condition for $E(\psi, \alpha)$ to be non-empty.

\begin{lemma}\label{overiff}
$\overline{E}(\psi)$ is non-empty if and only if $\psi$ is limsup-equivalent to a non-decreasing function.
\end{lemma}

\begin{proof}
For the ``only if" part, we assume that $\overline{E}(\psi)$ is non-empty. Take $x\in \overline{E}(\psi)$, namely,
\[
\limsup_{n \to \infty} \frac{\log a_1(x)+\cdots+\log a_n(x)}{\psi(n)} =1,
\]
and define $\phi: \mathbb{N} \to \mathbb{R}^+$ as $\phi(n):=\log a_1(x)+\cdots+\log a_n(x)+1$. Then $\phi: \mathbb{N} \to \mathbb{R}^+$ is non-decreasing and
\[
\limsup_{n \to \infty} \frac{\phi(n)}{\psi(n)} =1,
\]
which means that $\psi$ is limsup-equivalent to the non-decreasing function $\phi$.

For the ``if" part, we assume that $\psi$ is limsup-equivalent to a non-decreasing function $\phi: \mathbb{N} \to \mathbb{R}^+$.
Then
\begin{equation}\label{overlimsup}
\limsup_{n \to \infty} \frac{\phi(n)}{\psi(n)}=1.
\end{equation}
Since $\psi(n)/n \to \infty$ as $n \to \infty$, we can find a strictly increasing sequence $\{n_k\}$ such that for each $k \geq 1$,
\[
\frac{\psi(n)}{n} \geq k^2, \ \ \forall n \geq n_k.
\]
Let $\alpha_n: =k+1$ if $n_k \leq n <n_{k+1}$ with the convention $n_0\equiv1$. Then $\alpha_n \to \infty$ as $n \to \infty$ and
\begin{equation}\label{alphan}
\lim_{n \to \infty} \frac{\log\alpha_1 +\cdots+\log \alpha_n}{\psi(n)} = 0.
\end{equation}
Now define a new function $\widehat{\phi}: \mathbb{N} \to \mathbb{R}^+$ as $\widehat{\phi}(n):=\phi(n)+\log(\alpha_1\cdots \alpha_n)$. Then $\widehat{\phi}$ is non-decreasing and $\widehat{\phi}(n)/n \to \infty$ as $n \to \infty$.
Define a point $x\in (0,1)$ as
\[
\widehat{x}:=[a_1,a_2,\cdots,a_n,\cdots] \ \ \ \text{with}\ a_n= \lfloor e^{\widehat{\phi}(n)-\widehat{\phi}(n-1)}\rfloor,
\]
with the convention $\widehat{\phi}(0)=0$. Then $a_n(\widehat{x})=a_n$ and hence
\[
\lim_{n \to \infty} \frac{\log a_1(x)+\cdots+\log a_n(x)}{\widehat{\phi}(n)} =1
\]
since $\widehat{\phi}(n)/n \to \infty$ as $n \to \infty$. Combining this with \eqref{overlimsup} and \eqref{alphan}, we see that $\widehat{x}\in \overline{E}(\psi)$, and so $\overline{E}(\psi)$ is non-empty.
\end{proof}

\begin{lemma}\label{overnon-empty}
$\overline{E}(\psi)$ is always non-empty.
\end{lemma}

\begin{proof}
By Lemma \ref{overiff}, it is sufficient to show that $\psi$ is limsup-equivalent to a non-decreasing function. Indeed,
let $\phi(n):=\min_{k \geq n}\{\psi(k)\}$. Then $\phi$ is non-decreasing, $\phi(n) \to \infty$ as $n \to \infty$, $\phi(n) \leq \psi(n)$ and so
\[
\limsup_{n \to \infty} \frac{\phi(n)}{\psi(n)} \leq 1.
\]
In addition, we also obtain a useful observation:
\[
\phi(n) \neq \psi(n) \ \  \Longrightarrow \ \ \phi(n) = \min_{k \geq n}\{\psi(k)\} = \min_{k \geq n+1}\{\psi(k)\} = \phi(n+1).
\]
We claim that
\[
\phi(n) = \psi(n), \  \ \text{i.m.}\ n \in \mathbb{N}.
\]
In fact, if there exists $N \in \mathbb{N}$ such that $\phi(n) \neq \psi(n)$, i.e., $\phi(n) < \psi(n)$ for all $n \geq N$, then
\[
\phi(N)=\phi(N+1), \ \phi(N+1)=\phi(N+2),\cdots,
\]
which is in contradiction with the fact that $\phi(n)\to \infty$ as $n \to \infty$. Hence
\[
\limsup_{n \to \infty} \frac{\phi(n)}{\psi(n)} \geq 1.
\]
Therefore, $\psi$ is limsup-equivalent to the non-decreasing function $\phi$.
\end{proof}

Now we are ready to give the proof of Theorem \ref{OU} for the case $\overline{E}(\psi)$, which is divided into two parts: the upper bound and the lower bound of $\dim_{\rm H}\overline{E}(\psi)$.

{\bf Upper bound:}
For any $0<\varepsilon <1$, we have
\begin{equation}\label{oversubsetupper}
\overline{E}(\psi) \subseteq \left\{x\in (0,1): \Pi_n(x)\geq e^{(1-\varepsilon)\psi(n)},\ \text{i.m.}\ n \in \mathbb{N}\right\}.
\end{equation}
This leads to study the Hausdorff dimension of the limsup set.

\begin{lemma}\label{F}
Let $A \in (1,\infty)$. Write
\[
F(\psi):= \left\{x\in (0,1): \Pi_n(x)\geq A^{\psi(n)},\ \text{i.m.}\ n \in \mathbb{N}\right\}.
\]
Then
\[
\dim_{\rm H}F(\psi) = \frac{1}{b+1},
\]
where $b \in [1,\infty]$ is defined as in Theorem \ref{OU}.

\end{lemma}

\begin{proof}
The proof is divided into three parts: $b =1$, $1<b<\infty$ and $b=\infty$.

For the case $b=1$, since $\varphi(n)/n \to \infty$ as $n \to \infty$, we get that $F(\psi)$ is a subset of $\Pi_\infty$. By Lemma \ref{P},
\[
\dim_{\rm H}F(\psi) \leq \dim_{\rm H}\Pi_\infty =\frac{1}{2} = \frac{1}{b+1}
\]
For any $\varepsilon >0$, by the definition of $b$, we obtain $\psi(n) \leq (1+\varepsilon)^n$ for infinitely many $n$'s, and so
\[
\left\{x\in (0,1): \Pi_n(x) \geq A^{(1+\varepsilon)^n},\forall n\gg 1\right\} \subseteq F(\psi),
\]
It follows from Lemma \ref{Pac} that $\dim_{\rm H}F(\psi) \geq 1/(2+\varepsilon)$. Letting $\varepsilon \to 0^+$, we get the desired lower bound.

For the case $1<b<\infty$, let $0<\varepsilon <b-1$. By the definition of $b$, we have: (i) $\psi(n) \leq (b+\varepsilon)^n$ for infinitely many $n$'s; (ii) $\psi(n) \geq (b-\varepsilon)^n$ for sufficiently large $n$. Then
\[
F(\psi) \supseteq \left\{x\in (0,1): \Pi_n(x) \geq A^{(b+\varepsilon)^n},\forall n\gg 1\right\}
\]
and
\[
F(\psi) \subseteq \left\{x\in (0,1): \Pi_n(x)\geq A^{(b-\varepsilon)^n},\ \text{i.m.}\ n \in \mathbb{N}\right\}.
\]
Applying Lemma \ref{Pac}, we see that
\[
 \frac{1}{b+\varepsilon +1} \leq \dim_{\rm H}F(\psi) \leq \frac{1}{b-\varepsilon +1}.
\]
Since $\varepsilon$ is arbitrary, we obtain $\dim_{\rm H}F(\psi) = 1/(b+1)$.

For the case $b=\infty$, let $C>1$ be large, we have $\psi(n) \geq C^n$ for sufficiently large $n$, and so
\[
F(\psi) \subseteq \left\{x\in (0,1): \Pi_n(x)\geq A^{C^n},\ \text{i.m.}\ n \in \mathbb{N}\right\}.
\]
It follows from Lemma \ref{Pac} that $\dim_{\rm H}F(\psi) \leq 1/(C+1)$. Letting $C \to \infty$, we get that $\dim_{\rm H}F(\psi) =0$.
\end{proof}

Combining \eqref{oversubsetupper} and Lemma \ref{F}, we deduce that
\[
\dim_{\rm H}\overline{E}(\psi) \leq \frac{1}{b+1}.
\]

{\bf Lower bound:}
For the lower bound of $\dim_{\rm H}\overline{E}(\psi)$, Liao and Rams \cite[p.\,71--74]{LR} only write their proof for the case $b>1$,
but we remark that their method is indeed valid for $b \geq 1$. We will list the outline of their proof, for the sake of completeness.
For the case $b=\infty$, we have $\dim_{\rm H}\overline{E}(\psi) \geq 0$. So we assume that $1\leq b<\infty$.

Let $\phi(n):=\min_{k \geq n}\{\psi(k)\}$. As in the proof of Lemma \ref{overnon-empty}, we have seen that $\phi(n) \leq \psi(n)$ and
$\phi$ is non-decreasing. Moreover,
\begin{equation*}
\lim_{n \to \infty}\frac{\phi(n)}{n} =\infty.
\end{equation*}
Let $\varepsilon>0$ be fixed. Now we define a sequence $\{B_n\}_{n\geq 1}$ as follows:
\[
B_1:=e^{\phi(1)}  \ \ \ \text{and}\ \ \ B_n:=\min\left\{e^{\phi(n)}, B^{b+\varepsilon}_{n-1} \right\},\ \forall n\geq 2.
\]
Then: (i) $B_n\leq B_{n+1} \leq B^{b+\varepsilon}_n$ for all $n \geq 1$; (ii) $B_n \leq e^{\phi(n)} \leq e^{\psi(n)}$ for all $n \geq 1$; (iii) $B_n= e^{\psi(n)}$ for infinitely many $n$'s.
See Liao and Rams \cite[p.\,72]{LR} for the proofs.
Hence
\begin{equation}\label{Bnpro}
\limsup_{n \to \infty} \frac{\log B_{n+1}}{\log B_n} \leq b+\varepsilon \ \ \ \ \text{and} \ \ \ \ \limsup_{n \to \infty} \frac{\log B_{n}}{\psi(n)}=1.
\end{equation}
By the definition of $B_n$, we claim that
\begin{equation}\label{BN}
\lim_{n \to \infty} \frac{\log B_n}{n} =\infty.
\end{equation}
In fact, since $\lim_{n \to \infty} \phi(n)/n =\infty$, it follows from the definition of $B_n$ that
\[
\liminf_{n \to \infty} \frac{\log B_n}{n} = (b+\varepsilon)\cdot \liminf_{n \to \infty} \frac{\log B_{n-1}}{n},
\]
which means that\eqref{BN} holds. Write
\[
b_1:=B_1 \ \ \ \text{and} \ \ \ b_n:=\frac{B_n}{B_{n-1}}, \forall n \geq 2.
\]
Then $b_n \geq 1$ and $B_n=b_1\cdots b_n$. By the second equation of \eqref{Bnpro}, we have $\mathbb{E}(\{b_n\}) \subseteq \overline{E}(\psi)$, and so $\dim_{\rm H}\overline{E}(\psi) \geq \dim_{\rm H}\mathbb{E}(\{b_n\})$.
Combining this with \eqref{Bnpro}, \eqref{BN} and Lemma \ref{2S}, we deduce that
\[
\dim_{\rm H}\overline{E}(\psi)\geq \left(2+\limsup_{n \to \infty}\frac{\log b_{n+1}}{\log b_1+\cdots+\log b_n}\right)^{-1} \geq \frac{1}{b+1+\varepsilon}.
\]
Therefore, $\dim_{\rm H}\overline{E}(\psi) \geq 1/(b+1)$ as $\varepsilon$ is arbitrary.

\subsection{Hausdorff dimension of $\underline{E}(\psi)$}
Being similar to Lemma \ref{overiff}, we obtain an analogous result for $\underline{E}(\psi)$.

\begin{lemma}\label{underiff}
$\underline{E}(\psi)$ is non-empty if and only if $\psi$ is liminf-equivalent to a non-decreasing function.
\end{lemma}

\begin{proof}
The proof is very similar to that of Lemma \ref{overiff}.
For the ``if" part, we remark that if $\psi$ is liminf-equivalent to $\varphi$, then
\[
\lim_{n \to \infty}\frac{\varphi(n)}{n} =\infty.
\]
So there is no need to make a modification in constructing the point $\widehat{x}$ as in the proof of Lemma \ref{overiff}.
\end{proof}

The result of Lemma \ref{underiff} helps to show that $\underline{E}(\psi)$ is non-empty.

\begin{lemma}\label{undernon-empty}
$\underline{E}(\psi)$ is always non-empty.
\end{lemma}

\begin{proof}
Let $\varphi(n):=\max_{1\leq k\leq n}\{\psi(k)\}$. Then $\varphi$ is non-decreasing and $\varphi(n) \geq \psi(n)$, where the latter implies that
\[
\liminf_{n \to \infty} \frac{\varphi(n)}{\psi(n)} \geq 1.
\]
Note that
\[
\varphi(n) \neq \psi(n)\ \ \Longrightarrow \ \ \varphi(n) = \max_{1\leq k\leq n}\{\psi(k)\} = \max_{1\leq k\leq n-1}\{\psi(k)\} = \varphi(n-1),
\]
so we see that $\varphi(n) = \psi(n)$ for infinitely many $n$'s and then
\[
\liminf_{n \to \infty} \frac{\varphi(n)}{\psi(n)} \leq 1.
\]
Therefore, $\psi$ is liminf-equivalent to the non-decreasing function $\varphi$. By Lemma \ref{underiff}, $\underline{E}(\psi)$ is non-empty.
\end{proof}

We are now in a position to calculate the Hausdorff dimension of $\underline{E}(\psi)$. The proof is divided into two parts: the upper bound and the lower bound of $\dim_{\rm H}\underline{E}(\psi)$

{\bf Upper bound:}  For any $0<\varepsilon<1$, we obtain
\[
\underline{E}(\psi) \subseteq \left\{x\in (0,1): \Pi_n(x)\geq e^{(1-\varepsilon)\psi(n)},\ \forall n \gg 1\right\}.
\]
This leads to study the Hausdorff dimension of the latter set.

\begin{lemma}\label{Fwidehat}
Let $A \in (1,\infty)$. Write
\[
\widehat{F}(\psi):=\left\{x\in (0,1): \Pi_n(x)\geq A^{\psi(n)},\ \forall n \gg 1\right\}.
\]
Then
\[
\dim_{\rm H}\widehat{F}(\psi) = \frac{1}{B+1},
\]
where $B \in [1,\infty]$ is defined as in Theorem \ref{OU}.
\end{lemma}
\begin{proof}
The proof is very similar to that of Lemma \ref{F}, so the details are left to interested readers. See also the proof of Theorem 4.4 of \cite{FW}.
\end{proof}

By Lemma \ref{Fwidehat}, we deduce that
\[
\dim_{\rm H}\underline{E}(\psi) \leq \frac{1}{B+1}.
\]

{\bf Lower bound:} For the lower bound of $\dim_{\rm H}\underline{E}(\psi)$, we point out that the method of Liao and Rams \cite[p.\,75--78]{LR} does not work for the case $B=1$.
In particular, they need the following key condition in their proof:
\[
\liminf_{n\to\infty}\frac{\psi(1)+\cdots\psi(n)}{\psi(n)} <\infty,
\]
see Lemma 3.2 of \cite{LR}. However, it is impossible to be true for the power function $\psi(n) = n^p$ with $p>1$.
We will make a modification of their original proof to overcome these difficulties.

For the case $B=\infty$, it is easy to see that $\dim_{\rm H}\underline{E}(\psi) \geq 0$. Now we assume that $1\leq B<\infty$.
For any $\varepsilon >0$, since
\[
\log B =\limsup_{n \to \infty}\frac{\log \psi(n)}{n},
\]
we have $\psi(n) \leq (B+\varepsilon/2)^n$ for sufficiently large $n$. This implies that, for fixed $j \in \mathbb{N}$,
\begin{equation}\label{tendzero}
\psi(n)(B+\varepsilon)^{j-n} \leq (B+\varepsilon/2)^n(B+\varepsilon)^{j-n} = (B+\varepsilon)^{j}\left(\frac{B+\varepsilon/2}{B+\varepsilon}\right)^n \to 0\ \ \text{as}\ n \to \infty.
\end{equation}
Let $T_j = \sup_{k \geq 1} \big\{c_{j,k}\big\}$ for all $j \geq 1$, where $c_{j,k}$ is defined as
\[
c_{j,k} :=
\left\{
  \begin{array}{ll}
    e^{\psi(k)}, & \hbox{$1\leq k \leq j$;} \\
    e^{\psi(k)(B+\varepsilon)^{j-k}}, & \hbox{$k \geq j+1$.}
  \end{array}
\right.
\]
By \eqref{tendzero}, the supremum in the definition of $T_j$ is achieved. Moreover, $c_{j,k} = c_{j+1, k}$ for all $1\leq k \leq j$ and $(c_{j,k})^{B+\varepsilon} = c_{j+1,k}$ for all $k \geq j+1$, which yield that $T_j \leq T_{j+1} \leq T^{B+\varepsilon}_j$ for all $j \geq 1$.

We claim that
\begin{equation}\label{T}
\liminf_{j\to\infty}\frac{ \log T_j}{\psi(j)}=1.
\end{equation}
By the definition of $T_j$, we get that $T_j \geq c_{j,j}=e^{\psi(j)}$ for all $j \geq 1$, and so
\[
\liminf\limits_{j\to\infty}\frac{\log T_j}{\psi(j)}\geq1.
\]
For the opposite inequality, let $t_j := \min\{k \geq 1: c_{j,k} =T_j\}$; that is the smallest number $k$ for which $c_{j,k}$ achieves the supremum in the definition of $T_j$.
Then $t_j \leq t_{j+1}$, and $t_j \to \infty$ as $j \to \infty$. Moreover,
\begin{equation}\label{tj}
 c_{j,t_j} > c_{j,k},\  \forall1\leq k <t_j\ \ \ \ \ \text{and}\ \ \ \  \ c_{j,t_j} \geq c_{j,k},\ \forall  k >t_j.
\end{equation}
We will show that $T_{t_j} = e^{\psi(t_j)}$ in the following three cases.
\begin{itemize}
  \item For $t_j <j$, we see that $c_{t_j, t_j}=c_{j,t_j}$, $c_{j,k} =c_{t_j, k}$ for all $1\leq k <t_j$ and $c_{j,k}>c_{t_j,k}$ for all $k >t_j$. Combining this with \eqref{tj}, we get that
\[
c_{t_j, t_j}>c_{t_j,k}, \ \forall  1\leq k <t_j \ \ \ \ \text{and} \ \ \ \ c_{t_j, t_j} >c_{t_j,k}, \ \forall k >t_j,
\]
which gives $T_{t_j} = c_{t_j, t_j}=e^{\psi(t_j)}$ by the definition of $T_{t_j}$.

  \item For $t_j =j$, we have $T_{t_j}=T_j=c_{j,t_j} =c_{t_j,t_j}=e^{\psi(t_j)}$.

  \item For $t_j >j$, we deduce that $c_{t_j,k}=c_{j,k}$ for all $1\leq k\leq j$, $c_{t_j,k} = (c_{j,k})^{(B+\varepsilon)^{k-j}}$ for all $j <k \leq t_j$ and $c_{t_j,k} = (c_{j,k})^{(B+\varepsilon)^{t_j -j}}$ for all $k >t_j$. Combining this with \eqref{tj}, we see that
\[
c_{t_j,t_j} = (c_{j,t_j})^{(B+\varepsilon)^{t_j-j}} > c_{j,t_j}>   c_{j,k} = c_{t_j,k}, \ \ \ \forall 1\leq k\leq j,
\]
\[
c_{t_j,t_j} = (c_{j,t_j})^{(B+\varepsilon)^{t_j-j}} >  (c_{j,k})^{(B+\varepsilon)^{k-j}} =c_{t_j,k},\ \ \ \forall j< k<t_j
\]
and
\[
c_{t_j,t_j} = (c_{j,t_j})^{(A+\varepsilon)^{t_j-j}} \geq (c_{j,k})^{(B+\varepsilon)^{t_j-j}}=c_{t_j,k},\ \ \ \forall k>t_j,
\]
which implies that $T_{t_j} = c_{t_j, t_j}=e^{\psi(t_j)}$.
\end{itemize}
Then
\[
\liminf\limits_{j\to\infty}\frac{\log T_j}{\phi(j)}\leq \liminf\limits_{j\to\infty}\frac{\log T_{t_j}}{\phi(t_j)} =1
\]
and so \eqref{T} holds.

Define $c_1:=T_1$ and
\[
c_n := \frac{T_n}{T_{n-1}},\ \forall n\geq 2.
\]
Then $c_n \geq 1$, $T_n = c_1\cdots c_n$ and
\begin{equation}\label{subsetunder}
\liminf_{n \to \infty} \frac{\log c_1+\cdots+\log c_n}{\psi(n)} =1
\end{equation}
by \eqref{T}. Since $\psi(n)/n \to \infty$ as $n \to \infty$, we see that
\begin{equation}\label{demunder1}
\lim_{n \to \infty} \frac{\log c_1+\cdots+\log c_n}{n} =\infty.
\end{equation}
Note that $T_{n+1} \leq T^{B+\varepsilon}_n$, so we deduce that
\begin{equation}\label{demunder2}
\limsup_{n \to \infty} \frac{\log c_{n+1}}{\log c_1+\cdots+\log c_n} \leq B+\varepsilon -1.
\end{equation}
Applying the sequence $\{c_n\}$ to Lemma \ref{2S}, \eqref{subsetunder} implies that $\mathbb{E}(\{c_n\})$ is a subset of $\underline{E}(\psi)$.
In view of \eqref{demunder1} and \eqref{demunder2}, we conclude that
\[
\dim_{\rm H}\underline{E}(\psi) \geq \dim_{\rm H}\mathbb{E}(\{c_n\}) = \left(2+\limsup_{n \to \infty}\frac{\log c_{n+1}}{\log c_1+\cdots+\log c_n}\right)^{-1} \geq \frac{1}{B+1+\varepsilon}.
\]
Since $\varepsilon$ is arbitrary, we obtain the desired lower bound.

{\bf Acknowledgement:}
The research is supported by National Natural Science Foundation of China (Grant Nos.~11771153, 11801591, 12171107), Guangdong Natural Science Foundation (Grant No.\,2018B0303110005)
and Guangdong Basic and Applied Basic Research Foundation (Grant  No. 2021A1515010056).

\end{document}